\documentclass[12pt]{article}
\usepackage{color}
\usepackage{pdfsync}
\usepackage{amssymb}
\usepackage{eufrak}
\usepackage{amsmath}

\usepackage{color}

\begin{document}

\def\CX{{\mathbb C}}
\def\QX{{\mathbb Q}}
\def\NX{{\mathbb N}}
\def\ZX{{\mathbb Z}}
\def\PX{{\mathbb P}}
\def\GX{{\mathbb G}}
\def\Ga{{\GX_a}}
\def\Gm{{\GX_m}}
\def\ord{\mbox{ord }}
\def\GL{{\rm GL}}
\def\SL{{\rm SL}}
\def\gl{{\rm gl}}
\def\Scal{{\rm Scal}}
\def\diag{{\rm Diag}}
\def\d{{
\partial}}
\def\frakG{{\mathfrak G}}
\def\calD{{\mathcal D}}
\def\calC{{\mathcal C}}
\def\calO{{\mathcal O}}
\def\calP{{\mathcal P}}
\def\calL{{\mathcal L}}
\def\d{{
\partial}}
\def\dx{{
\partial_{x}}}
\def\dt{{
\partial_{t}}}
\def\Gal{{\rm Gal}}
\def\tbar{{t}}
\def\Hbar{{\overline{H}}}
\def\Gbar{{\overline{G}}}
\def\Gbaro{{\overline{G^0}}}
\def\vr{{\Vec{r}}}
\def\va{{\Vec{\mathbf a}}}
\def\vb{{\Vec{\mathbf b}}}
\def\vc{{\Vec{\mathbf c}}}
\def\vf{{\Vec{\mathbf f}}}
\def\vy{{\Vec{\mathbf y}}}
\def\vta{{\Vec{\tau}}}
\def\vt{{\Vec{t}}}
\def\sd{{\sigma\delta}}
\def\ord{{\rm ord}}
\def\pp{{\mathbb{P}}^1({\mathbb{C}})}




\date{}
\def\QED{\hbox{\hskip 1pt \vrule width4pt height 6pt depth 1.5pt \hskip 1pt}}
\newcounter{defcount}[section]
\setlength{\parskip}{1ex}
\newtheorem{thm}{Theorem}[section]
\newtheorem{lem}[thm]{Lemma}
\newtheorem{cor}[thm]{Corollary}
\newtheorem{prop}[thm]{Proposition}
\newtheorem{defin}[thm]{Definition}
\newtheorem{defins}[thm]{Definitions}
\newtheorem{remark}[thm]{Remark}
\newtheorem{remarks}[thm]{Remarks}
\newtheorem{ex}[thm]{Example}
\newtheorem{exs}[thm]{Examples}
\def\res{{\rm res}}

\def\semi{\hbox{${\vrule height 5.2pt depth .3pt}\kern -1.7pt\times $ }}

\newenvironment{prf}[1]{\trivlist
\item[\hskip \labelsep{\bf
#1.\hspace*{.3em}}]}{~\hspace{\fill}~$\square$\endtrivlist}
\newenvironment{proof}{\begin{prf}{Proof}}{\end{prf}}
 \def\square{\QED}
 \newenvironment{proofofthm}{\begin{prf}{Proof of Theorem~\ref{ext}}}{\end{prf}}
 \def\square{\QED}
\newenvironment{sketchproof}{\begin{prf}{Sketch of Proof}}{\end{prf}}

\numberwithin{equation}{section}

\title{Linear Algebraic Groups as Parameterized Picard-Vessiot Galois Groups}



\author{Michael F. Singer\footnote{Department of Mathematics, North Carolina
State University, Box 8205, Raleigh, North Carolina 27695-8205, email:\texttt{singer@math.ncsu.edu}. The  author was partially supported by NSF Grant CCF-1017217. 2010 Mathematics Subject Classification: 12H05, 20G15, 34M03, 34M15, 34M50}}


\date{}




\maketitle

\begin{abstract} We show that a linear algebraic group is the Galois group of a parameterized Picard-Vessiot extension of $k(x), \ x'=1$, for certain differential fields $k$,  if and only if its identity component has no one dimensional quotient as a linear algebraic group.
\end{abstract}
\section{Introduction} In the usual Galois theory of polynomial equations, one starts with a polynomial having coefficients in a field\footnote{In this paper, all fields considered are of characteristic zero.} $k$, forms a splitting field $K$ of this polynomial and then defines the Galois group of this equation to be the group of field automorphisms of $K$ that leave $k$ element-wise fixed. A natural inverse question then arises: {\it Given the field $k$, which groups can occur as Galois groups}. For example, if $k= C(x), C$ an algebraically closed field  and $x$ transcendental over $C$, any finite group occurs as a Galois group (Corollary 7.10,\cite{volklein}). In the Galois theory of linear differential equations, one starts with a homogeneous linear differential equation  with coefficients in a differential field $k$ with algebraically closed constants $C$, forms a  Picard-Vessiot extension $K$ (the analogue of a splitting field) and defines the Galois group of the linear differential equation to be the differential automorphisms of $K$ that leave $k$ element-wise fixed. This Galois group is a linear algebraic group defined over $C$ and one can again seek to determine which groups occur as the Galois group of a homogeneous linear differential equation over a given differential field. For example, if $k = C(x), C$ an algebraically closed field, $x' = 1$ and $c'=0$ for all $c \in C$, then any linear algebraic group occurs as Galois group of a Picard-Vessiot extension of $k$ (\cite{hartmann2005, hartmann2007} for proofs of this as well as  references to earlier work). Besides putting the Picard-Vessiot theory on a firm modern footing, Kolchin developed a generalization of Picard-Vessiot extensions called {\em strongly normal extensions} and developed a Galois theory for these fields (see \cite{DAAG} for an exposition and references to the original articles and \cite{kovacic03} for a reworking of this theory in terms of differential schemes). The Galois groups of these extensions can be arbitrary algebraic groups. Kovacic \cite{kovacic69, kovacic71} studied the general inverse problem in the context of  strongly normal extensions and showed that this problem can be reduced to the inverse problem for linear algebraic groups and for abelian varieties. If $k = C(x)$ as above, Kovacic showed  that any abelian variety can be realized  and, combining  this with the solution for linear algebraic groups described above, one sees  that any algebraic group defined over $C$ can be realized as a Galois group of a strongly normal extension of $C(x)$ (Kovacic also  solved the inverse problem for connected solvable linear algebraic groups and laid out a general plan for attacking the inverse problem for linear groups over arbitrary fields).\\[0.1in]
In \cite{Landesman}, Landesman developed a new Galois theory  generalizing Kolchin's theory of strongly normal extension to include, for example, certain differential equations that contain parameters. The Galois groups appearing here are differential algebraic groups (as in \cite{kolchin_groups}).   A special case was developed   in \cite{CaSi} where the authors consider parameterized {\em linear}  differential equations and discuss various properties of the associated Galois groups, named parameterized Picard-Vessiot groups or  PPV-groups for short. These latter groups are linear differential algebraic groups in the sense of Cassidy~\cite{cassidy1}, that is, groups of matrices whose entries belong to a differential field and satisfy a fixed set of differential equations.  The inverse problem in these theories is not well understood. Landesman showed  that any connected differential algebraic group is a Galois group in his theory over {\it some} differential field that may depend on the given differential algebraic group (Theorem 3.66, \cite{Landesman}). The analogue of the field $C(x)$ mentioned above is a field $k_0(x)$ with commuting  derivations $\Delta = \{\d_x,\d_1, \ldots, \d_m\}, \ m\geq 1,$ where $k_0$ is a differentially closed (see the definition below) $\Pi = \{\d_1, \ldots ,\d_m\}$-differential field,  $x$ is transcendental over $k_0$,  $\d_i(x) = 0$ for $i=1, \ldots ,m$ and   $\d_x$ is defined on $k$ by setting $\d_x(a) = 0$ for all $a \in k_0$ and $\d_x(x) = 1$. It is not known, in general,  which differential algebraic groups  appear as Galois groups in Landesman's theory over this field. In  \cite{Landesman} and \cite{CaSi}, it is shown  that the additive group $\Ga(k_0)$  cannot appear while any proper subgroup of these groups does appear as a Galois group (the same situation for $\Gm(k_0)$ is also described in \cite{Landesman})\footnote{There are other  Galois theories of differential equations  due primarily to Malgrange \cite{malgrange_galois_04}, Pillay \cite{pillaygalois2, pillaygalois1, pillaygalois3} and Umemura \cite{umemura_invitation}. In particular,  inverse problems are addressed in  \cite{pillaygalois1}.  We will not consider these theories here.}. \\[0.1in]
%
%
%
%
In this paper, further progress is made on the inverse problem for the parameterized Picard-Vessiot theory over the field $k=k_0(x)$ described in the previous paragraph.  In the following, I characterize those linear {\it algebraic} groups,  considered as linear {\it differential} algebraic groups, that can occur as PPV-groups of PPV-extensions of $k$ (under suitable hypotheses concerning $k$). Before I state the main result of this paper, I will recall some definitions. Although these definitions may be stated in more generality, I will state them relative to the field $k$ defined above.\\[0.1in]
The parameterized Picard-Vessiot theory (PPV-theory) considers linear differential equations of the form
\begin{eqnarray}\label{ppveqn}\d_xY &=& AY\end{eqnarray}
where $A \in \gl_n(k)$. In analogy to classical Galois theory and Picard-Vessiot theory, we consider fields, called {\em PPV-extensions of $k$}, that act as ``splitting fields'' for such equations. A PPV-extension $K$ of $k$ for (\ref{ppveqn}) is a $\Delta$-field $K$ such that
\begin{enumerate}
\item $K = k\langle Z \rangle$, the $\Delta$-field generated by the entries of a matrix $Z \in \gl_n(K)$ satisfying $\d_xZ = AZ, \ \det(Z) \neq 0$.
\item $K^{\d_x} = k^{\d_x} = k_0$, where for any $\Delta$-extension $F$ of $k$, $F^{\d_x} = \{c \in F \ | \ \d_xc=0\}$.
\end{enumerate}
A $\Pi$-field $E$ is said to be {\it differentially closed}  (also called {\it constrainedly closed}, see, for example  \S 9.1 of \cite{CaSi}) if for any $n$ and any set $\{P_1(y_1, \ldots, y_n), \ldots, P_r(y_1, \ldots ,y_n), \linebreak  Q(y_1, \ldots
,y_n)\} \subset E\{y_1, \ldots , y_n\}$, the ring of differential polynomials in $n$ variables,  if the system 
\[
\{P_1(y_1, \ldots, y_n)=0, \ldots , 
P_r(y_1, \ldots ,y_n)=0, Q(y_1, \ldots ,y_n)\neq 0\}\]
has a solution in some  differential field $F$ containing $E$, then it has a solution in $E$. In \cite{CaSi} (and more generally in \cite{HaSi08}), it is shown that under the assumption that $k_0$ is  differentially closed, then PPV-extensions exist and are unique up to $\Delta$-$k$-isomorphisms.  This hypothesis has been weakened   to non-differentially closed $k_0$ in \cite{GGO}  and \cite{wibmer}. In these papers the authors give conditions weaker than differential closure for the existence and uniqueness of PPV-extensions and and discuss the corresponding Galois theory.  Although some of our results remain valid under these weaker hypotheses, we will assume in this paper that $k_0$ is $\Pi$-differentially closed.  The set of field-theoretic automorphisms of $K$ that leave $k$ elementwise  fixed and commute with the elements of  $\Delta$ forms a group $G$ called the {\it  parameterized Picard-Vessiot group } (PPV-group)  of  (\ref{ppveqn}). One can show that for any $\sigma \in G$, there exists a matrix $M_\sigma \in \GL_n(k_0)$ such that $\sigma(Z) =(\sigma(z_{i,j})) = ZM_\sigma$. Note that $\d_x$ applied to an entry of such an $ M_\sigma$ is  $0$ since these entries are elements of $k_0$ but that such an entry need not be constant with respect to the elements of $\Pi$.   In \cite{CaSi}, the authors show that the map $\sigma \mapsto M_\sigma$ is an isomorphism whose image is furthermore a {\it linear differential algebraic group}, that is, a group of invertible matrices whose entries satisfies some fixed set of polynomial {\it differential}  equations (with respect to the derivations $\Pi = \{\d_1, \ldots, \d_m\}$) in $n^2$ variables. We say that a set $X \subset \GL_n(k_0)$ is {\it Kolchin-closed} if it is the zero set of such a set of polynomial differential equations.  One can show that the Kolchin-closed sets form the closed sets of a topology, called the {\it Kolchin topology} on $\GL_n(k_0)$ ({\it cf.} \cite{cassidy1, cassidy6, CaSi, kolchin_groups}). \\[0.1in]
One more definition is needed before stating the  main result of this paper, Theorem~\ref{thm1}.  A  $\Pi$-field $F$ is a {\it $\Pi$-universal field } if for any $\Pi$-field $E \subset F$, finitely differentially generated over $\QX$, any $\Pi$-finitely generated extension of $E$ can be differentially embedded over $E$  into $F$ (\cite{DAAG}, p.~133). Note that a universal field is differentially closed.
%
%

%

%


\begin{thm} \label{thm1}Let $k$ be as above and $G(k_0)$ the group of $k_0$-points of a linear algebraic group $G$ defined over $k_0$.
\begin{enumerate}
\item[1.] If $k_0$ is $\Pi$-differentially closed  and $G(k_0)$ is a PPV-group of a PPV-extension of $k$, then the identity component of $G$ has no quotient (as an algebraic group) isomorphic to the additive group $\Ga$ or the multiplicative group $\Gm$.
\item[2.] If $k_0$ is a $\Pi$-universal field  and the identity component of $G$ has no quotient (as an algebraic group) isomorphic to the additive group $\Ga$ or the multiplicative group $\Gm$, then $G(k_0)$ is a PPV-group of a PPV-extension of $k$. \end{enumerate}
\end{thm}
The remainder of this paper is organized as follows. Section~\ref{sec2}  contains the proof of Theorem 1.1.  In fact, I prove the stronger result  (Proposition~\ref{prop1}) that if $G$ is a linear {\it differential} algebraic group such that $G(k_0)$ is the PPV-group of a PPV-extension of $k$, then the identity component of $G(k_0)$ has no quotient (as a linear differential algebraic group) isomorphic to $\Ga(k_0)$ or $\Gm(k_0)$. 
In Section~\ref{sec3}, I show that a linear algebraic group contains a Kolchin-dense finitely generated subgroup if and only if it has no quotient (as an algebraic group) isomorphic to the additive group $\Ga$ or the multiplicative group $\Gm$. Theorem 1.2 then follows from the fact that a linear differential algebraic group containing  a Kolchin-dense finitely generated subgroup is a PPV-group of a PPV-extension of the field $k$ when $k_0$ is a $\Pi$-universal field.  This latter fact is proven in \cite{MiSi11} using analytic tools. In Section~\ref{sec4}, I show that Theorem 1.2 does not hold for general linear differential algebraic groups, that is, there is a connected linear differential algebraic group $G$ having neither $\Ga$ nor $\Gm$ as a quotient with  the further property that $G$ is not the PPV-group of any PPV-extension of $k$. Section~\ref{sec5} contains some final comments. \\[0.1in]
The author wishes to thank Phyllis Cassidy for helpful discussions concerning the content of this paper.

\section{Proof of Theorem 1.1}\label{sec2} In this section we will prove the   stronger result (Proposition~\ref{prop1}) that if $G$ is a linear {\it differential} algebraic group such that $G(k_0)$ is the PPV-group of a PPV-extension of $k$, then the identity component of $G(k_0)$ has no quotient (as a linear differential algebraic group) isomorphic to $\Ga(k_0)$ or $\Gm(k_0)$. Let $ k_0$ be a $\Pi$-differentially closed field  and let $G$ be a linear differential group defined over $k_0$ and let $G^0$ be its identity component in the Kolchin topology. 

\begin{lem}\label{lem0} The homomorphism $l\d_1:\Gm(k_0) \rightarrow \Ga(k_0)$ where $l\d_1(u) = \d_1(u)/u$ maps $\Gm(k_0)$ onto $\Ga(k_0)$.
\end{lem}
\begin{proof} Since $k_0$ is differentially closed, we need only show that for any $u \in k_0$, there is a $\Pi$-differential extension $F$ of $k_0$ such that $\d_1y = u y $ has a solution $y \neq 0$ in $F$. Let $\Pi_1 = \{\d_2, \ldots , \d_m\}$ and let $F$ be the $\Pi_1$-field $k_0\langle v \rangle$, where $v$ is a $\Pi_1$-differentially transcendental element. We extend the derivation $\d_1$ from $k_0$ to $F$ by setting $\d_1 v = uv,$ and $\d_1(\d_2^{i_2}\ldots \d_m^{i_m} v) = \d_2^{i_2}\ldots \d_m^{i_m}(\d_1v) = \d_2^{i_2}\ldots \d_m^{i_m}(uv)$.  With these definitions, $F$ becomes a $\Pi$-differential extension of $k_0$ and $y=v$ satisfies $\d_1y = u y $.\end{proof}

\begin{lem}\label{lem1}   If $G^0(k_0)$  has 
$\Gm(k_0)$ or $\Ga(k_0)$ as a homomorphic image (under a differential algebraic  homomorphism) and $G(k_0)$ is a PPV-group of a PPV-extension of   $k_0(x)$, then $\Ga(k_0)$ is a PPV-group of a PPV-extension of  a finite algebraic extension $E$ of $k_0(x)$.
\end{lem}
\begin{proof} I will show that  this result  follows from the Galois theory of parameterized linear differential equations (\cite{Landesman},\cite{CaSi}). Let $K$ be a PPV-extension of $k_0(x)$ having $G$ as its PPV-group. The fixed field $E$ of $G^0$ is a finite algebraic extension of $k_0(x)$. If  $G^0$ has $\Gm(k_0)$ as a homomorphic image under a differential homomorphism then composing  composing this homomorphism with $l\d_1:\Gm(k_0) \rightarrow \Ga(k_0)$ where $l\d_1(u) = \d_1(u)/u$,  Lemma~\ref{lem0} implies  that $\Ga(k_0)$ would also be a homomorphic image of $G^0(k_0)$ under a differential homomorphism. Therefore we shall only deal with this latter case. Let $\phi: G^0(k_0) \rightarrow \Ga(k_0)$ be a surjective differential algebraic homomorphism and let $H$ be its kernel.  The Galois theory (Theorem 9.5, \cite{CaSi}) implies that the fixed field of $H$ is a PPV-extension $F$ of $E$ whose PPV-group over $E$ is differentially isomorphic to $\Ga(k_0)$.  \end{proof}
The following lemma is the key to showing that $\Ga(k_0)$ is not a PPV-group over a finite algebraic extension of $k_0(x)$. 
\begin{lem}\label{lem2} Let $E$ be a finite algebraic extension of $k_0(x)$ and $f \in E$. Let $K$ be the PPV-extension of $k_0(x)$ corresponding to the equation \[\d_x y = f. \] Let $z \in K$ satisfy $\d_xz = f$. Then there exists a nonzero linear differential operator $L \in k_0[\d_1]$ and an element $g \in E$ such that 
\[L(z) = g.\]
\end{lem}
\begin{proof} The proof of this lemma is a slight modification of Manin's construction of the Picard-Fuchs equations (see Section 3, pp.~64-65 of the English translation of \cite{manin58}).  We shall use (as does Manin) ideas and results that appear in \cite{chevalley51}. In Ch.~VI, \S 7 of \cite{chevalley51}, Chevalley shows that $\d_1$ can be used to define a map $D$ on differentials of $E$ satisfying  $D(ydx) = (\d_1y)dx$. Furthermore, Theorem 13 of Ch.~VI, \S 7 of \cite{chevalley51} states that for any differential $\omega$ and any place $P$ of $E$, we have $\res_PD(\omega) = \d_1(\res_P\omega)$ (where $\res_P$ denotes the residue at $P$). Let $\alpha_1, \ldots , \alpha_m$ be the non-zero residues of $fdx$ and let $R \in k[\d_1]$ be a nonzero linear differential operator such that $R(\alpha_i) =0, i = 1, \ldots , m$\footnote{Let $C$ be the $\d_1$-constants of $k_0$ and $\beta_1, \ldots ,\beta_s$ a $C$-basis of the $C$-span of the $\alpha_i$'s. Let $R(Y) = wr(Y,\beta_1, \ldots , \beta_s)$ where $wr$ denotes the wronskian determinant. $R(Y)$ is a linear differential polynomial yielding the desired $R \in k_0[\d_1]$.}. We then have that for any place $P$, $\res_P(R(f)dx) = R(\res_P(fdx)) = 0$. Therefore $R(f)dx$ has residue $0$ at all places, that is, it is a differential of the second kind.  Note that $ \d_1^i(R(f))dx$ is also a differential of the second kind for any $i\geq1$. The factor space of differentials of the second kind by the space of exact differentials has dimension $2g$ over $k$, where $g$ is the genus of $E$ (Corollary 1,  Ch.~VI, \S 8,\cite{chevalley51}). Therefore there exist $v_{2g}, \ldots v_0 \in k_0$ such that
\[v_{2g}\d_1^{2g}(R(f))dx + \ldots +v_0R(f)dx = d\tilde{g} = \d_x\tilde{g} dx\]
for some $\tilde{g} \in E$. This implies that there exists a linear differential operator $L\in k_0[\d_1]$ such that 
\[L(f) = \d_x\tilde{g}.\] Furthermore, $\d_x(L(z)) = L(\d_xz) = L(f) = \d_x\tilde{g}$. Therefore $L(z) = g$ where $g = \tilde{g} +c$ for some $c \in k_0$.
\end{proof}

\begin{prop}\label{prop1} If $G$ is a linear differential algebraic group defined over $k_0$ such that $G^0(k_0)$ has  $\Gm(k_0)$ or $\Ga(k_0)$ as a quotient (as a linear differential group), then $G(k_0)$ cannot be a PPV-group of a PPV-extension of   $k_0(x)$.
\end{prop}
\begin{proof} Assume that $G(k_0)$ is a PPV-group of a PPV-extension of $k_0(x)$.  Lemma~\ref{lem1} implies that, in this case, $\Ga(k_0)$ is a PPV-group of a PPV-extension $K$ of $E$, where $E$ is a finite algebraic extension of $k_0(x)$.  From Proposition 9.12 of \cite{CaSi}, $K$ is the function field of a $\Ga(k_0)$-principal homogeneous space. The corollary to Theorem 4 of (Ch.~VII,\S 3, \cite{kolchin_groups}) implies that this principal  homogeneous space is the trivial principal homogeneous space and so $K = E\langle z \rangle$ where for any $\sigma \in  \Ga(k_0)$ there exists a $c_\sigma \in k_0$ such that $\sigma(z) = z +c_\sigma$.   In particular, $\sigma(\d_x z) = 
\d_x z$ for all $\sigma \in \Ga(k_0)$ and so $\d_x z = f \in E$. Lemma~\ref{lem2} implies that there exists a linear differential operator $L \in k_0[\d_1]$ and an element $g \in E$ such that 
$L(z) = g$.  For any $\sigma \in \Ga(k_0)$, we have $g = \sigma(g) = \sigma(L(z))  = L(\sigma(z)) = L(z+ c_\sigma) = g + L(c_\sigma)$ so $L(c_\sigma) = 0$.  This implies that the PPV-group of $K$ over $E$ is a proper subgroup of $\Ga(k_0)$, a contradiction.\end{proof}
Theorem 1.1 follows from Propostion~\ref{prop1} by noting that a linear algebraic group is {\it a fortiori} a linear differential algebraic group.
\section{Proof of Theorem 1.2}\label{sec3} The proof of Theorem 1.2 is inspired by  \cite{tretkoff79}. 
In this latter  paper, the authors mix analytic and algebraic facts to show that any linear algebraic group defined over $\CX$ is the Galois group of a Picard-Vessiot extension of $\CX(x)$. Their proof is based on the following facts:
\begin{enumerate}
\item Any linear algebraic group contains a finitely generated Zariski-dense subgroup.
\item Let $\PX^1(\CX)$ be the Riemann Sphere and $x_0, x_1, \ldots , x_n$ be distinct points of $\PX^1(\CX)$. If $\rho:\pi_1(\PX^1(\CX)\backslash\{x_1, \ldots , x_m\}, x_0) \rightarrow \GL_n(\CX)$ is a representation of the first homotopy group of the Riemann sphere with $m$ punctures, then there exists a linear differential equation 
\begin{eqnarray}\label{eqn1}\frac{dY}{dx} &= &AY, \mbox{ $A$ an $n\times n$ matrix with entries in $\CX(x)$}\end{eqnarray}
with only regular singularities having $\rho$ as its monodromy representation (for some choice of fundamental solution matrix).
\item If equation~\ref{eqn1} has only regular singular points, then for some choice of fundamental solution matrix, the Galois group of this equation is the smallest linear algebraic group containing the image of the monodromy representation.
\end{enumerate}
With these facts, the authors of \cite{tretkoff79} proceed as follows. Let $G \subset \GL_n(\CX)$ be a linear algebraic group. Using (1), there exist elements $g_1, \ldots g_m \in G$ that generate a Zarski-dense subgroup of $G$.  We can furthermore assume that the $g_i$ have been chosen so that $\prod_{i=1}^mg_i = 1$. Let $x_0, x_1, \ldots , x_n$ be distinct points of $\PX^1(\CX)$ and let $\gamma_i, i=1, \ldots , m$, be the obvious loops, starting and ending at $x_0$ that each enclose a unique $x_i$. The map $\rho:\gamma_i \mapsto g_i$ defines a homomorphism $\rho:\pi_1(\PX^1(\CX)\backslash\{x_1, \ldots , x_m\}, x_0) \rightarrow G \subset \GL_n(\CX)$. From (2), we can conclude that there is a linear differential equation (\ref{eqn1}) with only regular singular points having $\rho$ as its monodromy representation. From (3), we conclude that $G$ is the Galois group of this equation.\\[0.1in]
When one tries to mimic this proof in the parameterized case, one immediately is confronted with the fact that there are linear differential algebraic groups that have no finitely generated Kolchin-dense subgroups, that  is, the analogue  of (3) is no longer true. For example, as we have seen in the proof of Lemma~\ref{lem2}, if $g_1, \ldots , g_m$ are any  elements in $\Ga(k_0)$, there exists a linear differential operator $L \in k_0[\d_1]$ such that $L(g_i) = 0$ for $i =1, \ldots ,m$. This implies that any finitely generated subgroup of $\Ga(k_0)$ is contained in a proper Kolchin-closed subset and so cannot be Kolchin-dense in $\Ga(k_0)$. Nonetheless, analogues of facts (2) and (3) can be proven in the context of parameterized linear differential equations (see \cite{MiSi11}) and one can conclude  the following (Corollary 5.2, \cite{MiSi11}). Let $k_0$ be a $\Pi$-universal differential field  and let $k_0(x)$ be a differential field  as defined in the introduction.
\begin{prop}\label{mono}Let $G$ be a linear differential algebraic group defined over $k_0$ and assume that $G(k_0)$ contains a finitely generated subgroup  that is Kolchin-dense in $G(k_0)$. Then $G(k_0)$ is the PPV-group of a PPV-extension of $k_0(x)$.
\end{prop}
The assumption that $k_0$ is universal is forced on us, at present, because the analytic techniques used to prove this result do not let us control the algebraic nature of the coefficients appearing in the differential equation defining the PPV-extension. This forces us to assume that $k_0$ is ``sufficiently large''.\\[0.1in]
Therefore to prove Theorem 1.2 we need to show that under the stated hypotheses, $G$ contains a Kolchin-dense finitely generated subgroup. In fact, we show the following. Note that for the rest of this section, $k_0$ will denote a $\Pi$-differentially closed field.
\begin{prop}\label{fingen} Let $G \subset \GL_n$ be a linear algebraic group defined over $k_0$. The group $G(k_0)$ contains a Kolchin-dense finitely generated  subgroup if and only if the identity component $G^0(k_0)$ has no quotient isomorphic (as an algebraic group)  to $\Ga(k_0)$ or $\Gm(k_0)$.\end{prop}
To prove this result we will need the following three lemmas.
\begin{lem}\label{lem3} Let $G$ be a linear algebraic group defined over $k_0$ and $G^0$ be its identity component. $G(k_0)$ contains a Kolchin-dense finitely generated group if and only if $G^0(k_0)$ contains a Kolchin-dense finitely generated group.
\end{lem}
\begin{proof} Assume that $G^0(k_0)$ contains a Kolchin-dense group generated by $g_1, \ldots, g_s$.  Let $\{h_1, \ldots , h_t\}$ be a subset of $G(k_0)$ mapping surjectively onto $G(k_0)/G^0(k_0)$.  The set  \linebreak $\{g_1, \ldots , g_s,h_1, \ldots , h_t\}$  generates a group that is Kolchin-dense in $G(k_0)$.\\[0.1in]
 Assume that $G(k_0)$ contains elements $g_1, \ldots , g_s$ that generate a Kolchin-dense subgroup.  From (\cite{wehrfritz}, p.142) or (\cite{borel-serre}, lemme 5.11, p.152), one knows that any linear algebraic group $G(k_0)$, $k_0$ algebraically closed,    is of the form $HG^0(k_0)$ where $H$ is a finite subgroup of $G(k_0)$. Therefore we may write each $g_i$ as a product of an element of $H$ and an element of $G^0(k_0)$ and so  we may assume that there is a finite set $S=\{\tilde{g}_1, \ldots, \tilde{g}_t\} \subset G^0(k_0)$ such that the group generated by $S$ and $H$ is Kolchin-dense in $G(k_0)$.  Extending $S$ if necessary, we may assume that $S$ is stable under conjugation by elements of $H$ and therefore that the group generated by $S$ is stable under conjugation by the elements of $H$. An elementary topological argument shows that the Kolchin-closure $G'$ of the group generated by $S$ is also stable under conjugation by $H$.  Therefore $H\cdot G'$ forms a group. It is a finite union of Kolchin-closed sets, so it is also Kolchin-closed. It contains $H$ and $S$ so it must be all of $G(k_0)$.  Finally $G'$ is normal and of finite index in $G(k_0)$ so it must contain $G^0(k_0)$. Clearly $ G'\subset G^0(k_0)$ so $G^0(k_0) = G'$ and this shows that $G^0(k_0)$ is finitely generated. \end{proof}

\begin{lem}\label{lem4}Let $P\subset \GL_n$ be a connected semisimple linear algebraic group defined over $k_0$.  Then $P(k_0)$ contains a finitely generated Kolchin-dense subgroup. 
 \end{lem}
 \begin{proof} From Proposition 1 of \cite{tretkoff79} or Lemma 5.13 of  \cite{PuSi2003}, we know that a linear algebraic group contains a Zariski-dense finitely generated subgroup $H$.  We also know that $P$  contains a maximal torus $T$ of positive dimension. After conjugation, we may assume that $T$ is diagonal and that the projection onto the first diagonal entry is a homomorphism of $T$ onto $k_0^* = k_0\backslash \{0\}$.  Since $k_0$ is differentially closed, the derivations $\d_1, \ldots, \d_m$ are linearly independent so there exist nonzero elements $x_1, \ldots , x_m \in k_0$ such that $\det(\d_ix_j)_{1 \leq i,j\leq m} \neq 0$ (Theorem 2, p.~96, \cite{DAAG}).  For each $i = 1 ,\ldots , m$, let $g_i \in T$ be an element whose first diagonal entry is $x_i$. Let $P'$  be the Kolchin-closure of the group generated by $H$ and $\{g_1, \ldots , g_m\}$.  I claim $P'=P$.\\[0.1in]
 To see this note that since $P'$ contains $H$, $P'$ is Zariski-dense in $P$.  If $P' \neq P$, then  results of \cite{cassidy6} imply that there exist a nonempty subset $\Sigma \subset k_0\Pi$, the $k_0$ span of $\Pi$, such that  $P'$ is conjugate to a group of the form $P''(C)$ where $P''$ is a semisimple algebraic group defined over $\QX$ and $C = \{ c \in k_0 \ | \ \d c = 0 \mbox{ for all } \d \in \Sigma\}$. This implies that each element of $G$ has eigenvalues in $C$ and so, for each $x_i, \ \d(x_i) = 0$ for all $\d \in \Sigma$.  Yet, if $\d = \sum_{j=1}^m a_j \d_j$, not all $a_j$ zero and $\d(x_i) = 0$ for $i = 1, \ldots , m$, then $(a_1, \ldots a_m)X = (0, \ldots , 0)$ where $X = (\d_ix_j)_{1 \leq i,j\leq m}$.  This contradicts the fact that $\det X \neq 0$.  Therefore $P' = P$.
 \end{proof}
 \begin{lem}\label{lem5} Let $G(k) = P(k_0)\ltimes U(k_0)$ be a connected  linear algebraic group where $P(k_0)$ is a semisimple linear algebraic group and $U(k_0)$ is a commutative unipotent group, both defined over $k_0$.  If  $G(k_0)$ has no quotient isomorphic to  $\Ga(k_0)$, then $G(k_0)$ contains a Kolchin-dense finitely generated  subgroup. 
 \end{lem}
\begin{proof}  Note that $U(k_0)$ is isomorphic to $k_0^m$ for some $m$. Since $P$ acts on $U$ by conjugation, we may write $U = \oplus_{i=1}^m U_i$ where each $U_i$ is an irreducible $P$-module.  Furthermore, if the action of $P$ on some $U_j$ is trivial, then this $U_j$  would be of the form $\Ga(k_0)$ and we could write $P\ltimes U = (P\ltimes \oplus_{i\neq j}U_i)\times \Ga(k_0)$. This would imply that there is an algebraic morphism of $G(k_0)$ onto  $\Ga(k_0)$, a contradiction. Therefore we may assume the action of $P$ on each $U_i$ is nontrivial. Let $B$ be a Borel subgroup of $P$. From the representation theory of semisimple algebraic groups (Ch.13.3, \cite{humphreys}), we know that each $U_i$ contains a unique $B$-stable one-dimensional subspace corresponding to a weight $\lambda_i:B \rightarrow \Gm(k_0)$ (the highest weight of $U_i$). For each $i$, let $u_i$ span this one-dimensional space. We claim that the $P(k_0)$-orbit of $u_i$ generates a group that equals $U_i(k_0)$. Note that since  $B(k_0)$ is connected and $\lambda_i$ is not trivial, we have the $P(k_0)$-orbit of $u_i$ contains $\Gm(k_0)u_i$. Since $U_i$ is an irreducible $P(k_0)$-module, there exist $g_1, \ldots , g_s \in P(k)$, such that $g_1u_i g_1^{-1}, \ldots g_su_i g_s^{-1}$ span $U_i$. Since $g_j(\Gm(k_0)u_i)g_j^{-1} = \Gm(k_0)(g_ju_ig_j^{-1})$ for $j = 1, \ldots , s$, we have that the $P(k_0)$-orbit of $u_i$ generates all of $U_i$. \\[0.1in]
Now Lemma~\ref{lem4} asserts that there exists a finite set   $S\subset P(k_0)$ that  generates a Kolchin-dense subgroup of $P(k_0)$. We then have  that $S\cup\{u_i\}_{i=1}^m$ generates a Kolchin-dense subgroup of $G(k_0)$.  \end{proof}
\noindent {\it Proof of Proposition~\ref{fingen}.} Assume that $G(k_0)$ contains a Kolchin-dense finitely generated subgroup. Lemma~\ref{lem3} implies that $G^0(k_0)$ also contains a Kolchin-dense finitely generated subgroup.  If there is an algebraic morphism of $G^0(k_0)$ onto $\Gm(k_0)$ or $\Ga(k_0)$ then, in the first case,  composing this with the differential algebraic morphism $l\d_1:\Gm(k_0) \rightarrow \Ga(k_0)$ where $l\d_1(u) = \d_1(u)/u$, we would  have a differential algebraic morphism of $G^0(k_0)$ onto $\Ga(k_0)$.  Therefore, in either case we have a differential homomorphism of $G^0(k_0)$ onto $\Ga(k_0)$. This implies that $\Ga(k_0)$ would contain a  Kolchin-dense finitely generated subgroup.  On the other hand, any finite set of elements of $\Ga(k_0)$ satisfy a linear differential equation and so could not generate a Kolchin-dense subgroup.  Therefore there is no algebraic morphism of $G^0(k_0)$ onto $\Gm(k_0)$ or $\Ga(k_0)$.\\[0.1in]
 Assume that there is no algebraic morphism of $G^0(k_0)$ onto $\Gm(k_0)$ or $\Ga(k_0)$. Lemma~\ref{lem3} implies that it is enough to show that $G^0(k_0)$ contains a Kolchin-dense finitely generated group. We may write $G^0 = P \ltimes R_u$ where $P$ is a Levi subgroup and $R_u$ is the unipotent radical of $G$ (\cite{humphreys}, Ch.30.2). \\[0.1in]
 We first claim that $P$ must be semisimple. We may write $P = (P,P) Z(P)$ where $(P,P)$ is the derived subgroup of $P$ and $Z(P)$ is the center of $P$. Furthermore, $Z(P)^0$ is a torus (\cite{humphreys}, Ch.27.5). We therefore have a composition of surjective morphisms 
 \[G^0 \rightarrow G^0/R_u \simeq P \rightarrow P/(P,P) \simeq Z(P)/(Z(P) \cap (P,P)). \] 
  Since $G$ is connected, its image lies in the image of $Z^0(P)$ in   $Z(P)/(Z(P) \cap (P,P)$ and therefore is a torus. This  torus,  if not trivial, has a quotient isomorphic to $\Gm$. This would yield a homomorphism of $G^0(k_0)$ onto $\Gm(k_0)$ and, by assumption, this is not possible. Therefore $Z^0(P)$ is trivial. Since $G^0$ is connected we must have $Z(P) \subset (P,P)$. Therefore  $P = (P,P)$ and is therefore semisimple.\\[0.1in]
 We shall now show that it suffices to prove that $G^0(k_0)$ contains a Kolchin-dense finitely generated subgroup under the assumption that $R_u$ is commutative.  In \cite{kovacic69}, Kovacic shows (\cite{kovacic69},Lemma 2): {\em Let $G$ be an abstract group, $H$ a subgroup and $N$ a nilpotent normal subgroup of $G$. Suppose $H\cdot (N,N) = G$. Then $H = G$.} Therefore, if we can find a Kolchin-dense finitely generated subgroup  of the $k_0$-points of $G^0/(R_u,R_u) \simeq P\ltimes (R_u/(R_u,R_u))$, then the preimage of this group under the homomorphism $G^0 \rightarrow   G^0/(R_u,R_u)$ generates a Kolchin-dense subgroup of $G(k_0)$. \\[0.1in]
 Therefore, we need only consider connected groups satisfying the hypotheses of Proposition~\ref{fingen} and  of the form $P(k_0)\ltimes U(k_0)$, where $P$ is semisimple and $U$ is a commutative unipotent group.  Lemma~\ref{lem5} guarantees that such a group has a finitely generated Kolchin-dense subgroup. \hfill  $\QED$\\[0.1in]
 Proposition~\ref{mono} and Proposition~\ref{fingen} imply that Theorem 1.2 is true.\\[0.1in]
 It would be of interest to find a purely algebraic proof of Theorem 1.2 that would perhaps also show that this result is true when we weaken the hypotheses to assume that $k_0$ is only differentially closed (or even just algebraically closed).   Furthermore, the relation between the conditions that  a group contains a Kolchin-dense finitely generated subgroup and  that the group appears as a PPV-group of a PPV-extension of $k_0(x)$ should be further studied.  I know of no example of a linear differential algebraic group that is a PPV-group of a PPV-extension of $k_0(x)$ and that does not contain a Kolchin-dense finitely generated subgroup. See Section~\ref{sec5} for  further discussion concering this. 
 
   \section{An Example}\label{sec4}
In this section we give an example that shows that Theorem 1.2 and Proposition~\ref{fingen} are not true for linear differential algebraic groups in general.\\[0.1in]
Let $k_0$ be an ordinary  differentially closed field with derivation $\d_1$ and let
\[G = \{ \left(\begin{array}{cc} 1& 0 \\ a & b \end{array} \right) \ | a, b \in k_0 , \ b\neq 0, \ \d_1 b = 0 \} \simeq G_1 \rtimes G_2\]
where  
\begin{eqnarray*}
  G_1&= & \{ \left(\begin{array}{cc} 1& 0 \\ a& 1 \end{array} \right) \ | a \in k_0  \} \simeq \Ga(k_0)\\
  G_2& = & \{ \left(\begin{array}{cc} 1& 0 \\ 0 & b \end{array} \right) \ | b \in k_0 , \ b\neq 0, \ \d_1b = 0 \}
  \simeq  \Gm(C) 
\end{eqnarray*}
where $C = \{c \in k_0 \ | \ \d_1c = 0\}$. Let $k = k_0(x)$ be a $\Delta = \{ \d_x, \d_1\}$-field as in the introduction.\\[0.1in]
If $k_0$ is a universal field then Proposition~\ref{mono} implies that statement 1.~ of the following proposition follows from statement 4.  We do not make this assumption.
\begin{prop}\label{exprop} $\mbox{  } $ \begin{enumerate}
\item[1.] $G(k_0)$ contains no Kolchin-dense finitely generated  subgroup.
\item[2.] $G$  is Kolchin-connected.
\item[3.] There is no surjective differential algebraic  homomorphism of $G(k_0)$ onto $\Ga(k_0)$ or $\Gm(k_0)$. 
\item[4.] $G(k_0)$ is not a PPV-group of a PPV-extension of $k_0(x)$.
\end{enumerate}
\end{prop}
\begin{proof}
1.~To see that $G(k_0)$ contains no Kolchin-dense finitely generated  subgroup, note that any element of $G(k_0)$ can be written as a product of an element of $\Ga(k_0)$ and $\Gm(C)$. Therefore it is enough to show that any set of elements of the form
\[\left(\begin{array}{cc} 1& 0 \\ a_1 & 1 \end{array} \right), \ldots , \left(\begin{array}{cc} 1& 0 \\ a_n & 1 \end{array} \right),\left(\begin{array}{cc} 1& 0 \\ 0 & b_1 \end{array} \right), \ldots , \left(\begin{array}{cc} 1& 0 \\ 0 & b_m \end{array} \right)\]
with the $a_i \in k_0$ and the $b_i \in C$ do not generate a Kolchin-dense subgroup of $G$. Let $H$ be the group generated by these elements and let $L \in k[\d_1]$ be a nonzero differential operator such that $L(a_i) = 0$ for all $i = 1, \ldots , n$.  A calculation shows that any element of $H$ is of the form 
\[\left(\begin{array}{cc} 1& 0 \\ c_1a_1+ \ldots +c_na_n & b \end{array} \right)\]
with $b \in k_0$ and  the $c_i \in C$.  Therefore $H$ is a subgroup of
\[ \{\left(\begin{array}{cc} 1& 0 \\ a & b \end{array} \right) \ | \ L(a) = 0, \d_1 b = 0, b\neq 0\}\]
which is a proper Kolchin-closed subgroup of $G$. \\[0.1in]
2.~Assume that $G$ is not Kolchin-connected and let $G^0$ be the identity component in the Kolchin topology. One has that $G^0$ is normal  and of finite index in $G$. Furthermore, $G/G^0$ is again a linear differential algebraic group and $\pi:G \rightarrow G/G^0$ is a differential algebraic homomorphism. Since $G_1 \simeq \Ga(k_0)$ is a divisible group and any homomorphism of a divisible group into a finite group is trivial, we have that   $G_1$ is contained in the kernel of $\pi$. This implies that $\pi$ induces a differential algebraic homomorphism $\pi^*$ from $G_2\simeq \Gm(C)$ to a finite group.  Since the elements of $\Gm(C)$ have constant entries, $\pi^*$ is really an algebraic homomorphism of $\Gm(C)$ into a finite group. Since $\Gm(C)$ is connected in the Zariski topology, this homomorphism is trivial.  Therefore $G = G^0$.\\[0.1in]
3.~Since $\Ga(k_0)$ is a differential homomorphic image of $\Gm(k_0)$, it suffices to show that there is no  surjective differential algebraic  homomorphism of $G(k_0)$ onto $\Ga(k_0)$. Assume not and let $\phi:G(k_0)\rightarrow \Ga(k_0)$. Restricting $\phi$ to $G_2$ yields an {\em algebraic} homomorphism of $\Gm(C)$ into $\Ga(k_0)$. Since algebraic homomorphisms preserve the property of being semisimple, we must have that $G_2 \subset \ker \phi$. Therefore for any $a \in k$ and any $b \in C^*$, we have 
\[ \left(\begin{array}{cc} 1& 0 \\ a & 1 \end{array} \right)\left(\begin{array}{cc} 1& 0 \\ 0 & b \end{array} \right)\left(\begin{array}{cc} 1& 0 \\ -a & 1 \end{array} \right) = \left(\begin{array}{cc} 1& 0 \\ a-ba & b \end{array} \right)\in \ker \phi .\]
For any $\tilde{a} \in k_0$ and $1\neq b \in C$ there exists a $a \in k_0$ such that $a-ba = \tilde{a}$, so $\ker \phi$ contains all elements of the form 
\[\left(\begin{array}{cc} 1& 0 \\ \tilde{a} & b \end{array} \right)\]
$a \in k_0$, $1\neq b \in C$.   Since $G_2 \subset \ker \phi$ as well, we have that $G \subset \ker \phi$, a contradiction.\\[0.1in]
4.~
Assume  $G$ is a PPV-group of a PPV-extension $K$ of $k=k_0(x)$. The field $K$ is then a PPV-extension corresponding to a second order linear differential equation, in $\d_x$, $L(Y) = 0$. There are  elements $u,v \in K$ forming a $k_0$-basis of the solutions space such that for any $\sigma \in G, \ \sigma =  \left(\begin{array}{cc} 1& 0 \\ a & b \end{array} \right)$, we have $\sigma(u) = u+av$ and $\sigma(v) = bv$.   We have $K = k<u,v>_\Delta$, the $\Delta =\{\d_x, \d_1\}$-differential field generated by $u$ and $v$. I will describe the structure of $K$ in more detail.
\begin{enumerate}
\item Let $E = k(v)$. Since for any $\sigma \in G$ there exists a $b_\sigma \in C$ such that $\sigma(v) = b_\sigma v$, we have that $\d_xv/v \in k$ and $\d_1 v/v \in k$.  Furthermore, $E$ is the fixed field of $G_1$  and the PPV-group of $E$ over $F$ is 
$G/G_1 \simeq G_2$.
\item We may write $K = E<w>_\Delta$ where $w = \frac{u}{v}$.  For any $\sigma =  \left(\begin{array}{cc} 1& 0 \\ a & 1 \end{array} \right) \in G_1$ we have $\sigma (w) = w + a$. Therefore  $\sigma(\d_x w) =\d_x w$, so $\d_x w \in E$.
\item For any $\sigma = \left(\begin{array}{cc} 1& 0 \\ 0 & b \end{array} \right) \in G_2$ we have $\sigma(\d_x w) = \d_x(\sigma(w)) = \d_x(\frac{1}{b}w) = \frac{1}{b} \d_x w$.  Since $\sigma(v) = bv$, we have  $\sigma(\d_xw \cdot v) = \d_1 w \cdot v$.  This implies $\d_1 w = r/v$ for some $r \in k$. In particular, we may write $E = k(\d_x w)$ and that $\d_x(\d_xw) /\d_x w = A \in k$ and $\d_1(\d_xw) /\d_x w = B \in k$. 
\end{enumerate}
Summarizing, we have
\[k = k_0(x) \subset E =k(\d_x w) \subset E<w>_\Delta = K\]
where \[\frac{\d_x(\d_xw)}{\d_x w }= A \in k\mbox{  and } \frac{\d_1(\d_xw)}{\d_x w} = B \in k.\]
I now claim that there exists an element $h \in k_0(x)$ and a nonzero operator
\[L = \sum_{i=0}^M \alpha_i \d_1^i\in k_0[\d_1]\] such that 
\[L(\d_x w) = (\d_x h + hA)\d_xw.\]
Let us assume that this last claim is true. We then would have that
\[\d_x(L(w) - h\d_xw) = L(\d_x w) - \d_xh \d_xw - hA\d_xw = 0.\]
Therefore $L(w) = h\d_xw + c \in E$ for some $c \in k_0$. In particular $L(w)$ is left fixed by all $\sigma \in H \simeq \Ga(k_0)$. This means that 
\[L(w) = \sigma(L(w)) = L(w + a_\sigma) = L(w) + L(a_{\sigma})\]
where $\sigma = \left(\begin{array}{cc} 1& 0 \\ a_\sigma & 1 \end{array} \right)$ and so $L(a) = 0$ for all $a \in k_0$, a contradiction.\\[0.1in]
We shall now show that the claimed $L$ and $h$ exist. If $A \in k_0$, then we let $L = A$ and $h = 1$. This yields $L(\d_xw) = A\d_xw$ as desired. Therefore we may assume $A \in k_0(x)$ but $A \notin k_0$ and so $A$ has poles.  Let $x_1, \ldots , x_p \in \PX^1(k_0)$ include all the poles of $A$ and $B$ (including $x_p =  \infty$). Select $n \in \NX$ so that $A$ and $B$ have poles of orders at most $n$ at each of these points. Select integers $M$ and $N$ such that 
\[M > np \mbox{ and } N> n(M-1)\]
(the reason for this choice will be apparent later). Let
\begin{eqnarray*}
L & = & \sum_{i=0}^M \alpha_i\d_1^i\\
h & = &\beta_{p,0} + \beta_{p,1}x + \ldots + \beta_{p,N}x^N + \sum_{i=1}^{p-1}\sum_{j=1}^N \frac{\beta_{i,j}}{(x-x_i)^j} 
\end{eqnarray*}
where the $\alpha_i$ and the $\beta_{i,j}$ are indeterminates. We shall show that the condition ``$L(\d_xw) = (\d_x h + hA)\d_xw$" forces these indeterminates to satisfy an undertedermined system of linear equations over $k_0$ and so there will always be a way to select elements of $k_0$ (not all zero) satisfying this system. Furthermore, we will show that not all the $\alpha_i$ can be zero. To see this, we shall look at the expressions $L(\d_x w)$ and $(\d_x h + hA)\d_1w$ separately.\\[0.1in]
$\underline{L(\d_xw):}$ Note that 
\begin{eqnarray*}
\d_1(\d_x w) & = & B \d_xw\\
\d_1^2(\d_x w) & = & \d_1 B \d_xw + B^2 \d_xw\\
\vdots \ \ \ & \vdots & \ \ \ \vdots \\
\d_1^i(\d_x w) & = & R_i \d_xw
\end{eqnarray*}
where  $R_i \in k_0(x)$ and $R_0 = 1, R_1 = B, R_{i+1} = \d_1 R_i + BR_i$.  We furthermore have 
\[ L(\d_x w) = (\sum_{i=0}^M \alpha_iR_i) \d_x w . \] 
Each $R_i$ has poles only at $x_1, \ldots , x_p$. Furthermore the order of a poles of each $R_i$ is at most $in$. 
Therefore the partial fraction decomposition of $\sum_{i=0}^M \alpha_iR_i$  is of the form
\begin{eqnarray}\label{gammaeqn}
\Gamma_{p,0} + \Gamma_{p,1}x + \ldots + \Gamma_{p,Mn}x^{Mn} + \sum_{i=1}^{p-1}\sum_{j=1}^{Mn} \frac{\Gamma_{i,j}}{(x-x_i)^j}& &
\end{eqnarray}
where the $\Gamma_{i,j}$ are linear forms in the $\{\alpha_r\}_{r=0}^M$. 
\\[0.1in]
$\underline{(\d_x h + Ah)\d_xw:}$ Once again  $\d_x h + Ah$ has  poles only at $x_1, \ldots , x_p$. The order at any of  these poles is at most $n + N$.  Therefore  the partial fraction decomposition of $\d_x h + Ah$ is of the form:
\begin{eqnarray}\label{lambdaeqn}
\Lambda_{p,0} + \Lambda_{p,1}x + \ldots + \Lambda_{p,Mn}x^{n+N} + \sum_{i=1}^{p-1}\sum_{j=1}^{n+N} \frac{\Lambda_{i,j}}{(x-x_i)^j}&&
\end{eqnarray}
where the $\Lambda_{i,j}$ are linear forms in the $\{\beta_{r,s}\}$ with coefficients in $k_0$. 
\\[0.1in]
The equation $L(\d_xw) = (\d_x h + hA)\d_xw$ implies that expression~(\ref{gammaeqn}) equals expression~(\ref{lambdaeqn}). 
When we  equate coefficients of powers of $x$  in these two expressions, we will produce   $n+N+ 1$ homogeneous linear equations in the variables $\{\alpha_r\}$  and $\{\beta_{r,s}\}$ (note that our assumption $N > n(M-1)$ implies that $N+n>Mn$). For each $x_i, \ i=1, \ldots, p-1$, equating the coefficients of powers of $x-x_i$ in these two expressions yields $N+n$ homogeneous linear equations in the variables $\{\alpha_r\}$  and $\{\beta_{r,s}\}$.  In total, equating coefficients of like terms yields $p(n+N) + 1$ homogenous linear equations in the $\{\alpha_r\}$  and $\{\beta_{r,s}\}$. The total number of  the $\{\alpha_r\}$  and $\{\beta_{r,s}\}$ is $M+1 + Np+1$. Because we have selected  $M >np$, we have \[M+1 + Np + 1 > np + 1 + Np + 1 = p(n+N) + 2\]
and this exceeds the number of equations.  Therefore we can find $\{\alpha_r\}$  and $\{\beta_{r,s}\}$ in $k_0$, not all zero,  that satisfy these equations.\\[0.1in]
 We shall now show that in any such choice, not all the $\alpha_r$ are zero. Assume all the $\alpha_r$ are zero.  In this case we would have $\d_xh + Ah = 0$.  This implies that $\d_x(h\d_x w ) = 0$ and so $\d_x w \in k_0(x)$.  This contradicts the fact that the field $E$ in our tower $k \subset E \subset K$ is a proper extension of $k$.
 \end{proof}
 It is interesting to contrast  the group $G$ above with  the slightly larger group 
 \begin{eqnarray*}
  G'&= & \{ \left(\begin{array}{cc} 1& 0 \\ a & b \end{array} \right) \ | a, b \in k_0 , \ b\neq 0, \ \d_1(\frac{\d_1 b}{b}) = 0 \}\\
  & = &   G_1' \rtimes G_2'
  \end{eqnarray*}
  where 
   \begin{eqnarray*}
  G_1'&= & \{ \left(\begin{array}{cc} 1& 0 \\ a& 1 \end{array} \right) \ | a, \in k_0  \} \simeq \Ga(k_0)\\
  G_2'& = & \{ \left(\begin{array}{cc} 1& 0 \\ 0 & b \end{array} \right) \ | b \in k_0 , \ b\neq 0, \ \d_1(\frac{\d_1 b}{b}) = 0 \}
  \simeq \{b\in \Gm(k_0) \ | \d_1(\frac{\d_1 b}{b}) = 0 \}
\end{eqnarray*}
Note that $G \subset G'$.   I will show that, $G'$  contains a Kolchin-dense finitely generated subgroup. To see this let 
\[ g = \left(\begin{array}{cc} 1& 0 \\ 1 & 1 \end{array} \right) \ \ \ \ \ \mbox{ and } \ \ \ \ \ h= \left(\begin{array}{cc} 1& 0 \\ 0 & e\end{array} \right)\]
where $0 \neq e \in k_0$ satisfies $\d e = e$.  I will show that the Kolchin closure of the group generated by these two elements is $G'$.\\[0.1in]
For $n \in \ZX$,  
\[h^{n}gh^{-n} = \left(\begin{array}{cc} 1& 0 \\ e^n & 1 \end{array} \right) \]
lies in $G_1'$. Since $S = \{e^n \ | \ n \in \ZX \}$ is a set of elements linearly independent over the $\d_1$-constants of $k_0$, $S$ cannot be a subset of  the solution space of a nonzero homogeneous linear differential equation with coefficients in $k_0$.  Therefore the Kolchin-closure of the group generated by $g$ and $h$ contains $G_1'$.\\[0.1in]
Let $H$ be the Kolchin-closure of the group generated by $h$. We will identify this with a subgroup of $\{b\in \Gm(k_0) \ | \d_1(\frac{\d_1 b}{b}) = 0 \} \subset \Gm(k_0)$. One sees that $H$ is Zariski-dense in $\Gm(k_0)$ and so, by (Corollary 2, p.~938,\cite{cassidy1}),  there exists a set  $\calL$ of linear operators in $\d_1$ with coefficients in $k_0$ such that $H = \{y \in \Gm(k_0) \ | \  L(\frac{\d_1 y}{y}) = 0, \forall \ L \in \calL\}$. We may assume that $\calL$ is a left ideal in  the ring $k_0[D]$ of linear operators. Since this ring  is a right euclidean domain, we have that $\calL$ is generated by a single element $R$. Since $H \subset G_1'$ we have $R$ divides the operator $D$ on the right. Therefore $R$ must equal $D$ and so $H = G_1'$. Therefore  the Kolchin-closure of the group generated by $g$ and $h$ contains $G_1'$ and $G_2'$ and so must equal $G'$. \\[0.1in]
When  $k_0$ is a universal field then Proposition~\ref{mono} implies that $G'$ is a PPV-group of some PPV-extension of $k$. In fact, as noted in Section 7 of \cite{CaSi},  when $k_0$ is the differential closure of $\CX(t)$ where $\d_1 = \frac{\d}{\d t}$, then $G'$ is the PPV-group of 
\begin{eqnarray}\label{gammafnceqn}
\d_x^2 y -\frac{t-1-x}{x}\d_xy  &= & 0 . 
\end{eqnarray} Knowing that $G'$ is the PPV-group of  equation (\ref{gammafnceqn}), one can also deduce from the results of \cite{dreyfus} (see the final comments below) that $G'$ contains a Kolchin-dense finitely generated  subgroup.


\section{Final Comments}\label{sec5} The results of Section~\ref{sec2} and Section~\ref{sec3} imply that if $k_0$ is a universal field and $G$ is a linear algebraic group defined over $k_0$, then $G(k_0)$ is a PPV-groups of a PPV- extension of $k$ if and only if $G(k_0)$ contains a Kolchin-dense finitely generated subgroup.  One can ask if this result is true for arbitrary linear differential algebraic groups. I know of no counterexamples. Of course, the implication in one direction is true since Proposition~\ref{mono} states that  a sufficient condition for $G(k_0)$ to be a PPV-group of a PPV-extension of $K$ is that $G(k_0)$ contains a Kolchin-dense finitely generated subgroup. A recent result of Dreyfus~\cite{dreyfus} bears on the implication in the other direction.  In \cite{dreyfus}, the author proves a generalization of the Ramis Density Theorem \cite{martinet_ramis} by showing  that the local parameterized Picard-Vessiot group of a parameterized linear differential equation with a fixed singular point is the Kolchin-closure of the group generated by an element called the formal parameterized monodromy, a finite number of elements called the (parameterized) Stokes operators and the constant points of a linear {\em algebraic} group called the parameterized exponential torus. As is pointed out in \cite{dreyfus},  this latter group is finitely generated. Using this result, one can further conclude that at least in the case of parameterized equations with fixed singular points, one has that the PPV-group of a parameterized linear differential equation contains a Kolchin-dense finitely generated subgroup.


\bibliographystyle{amsplain}
\newcommand{\SortNoop}[1]{}\def\cprime{$'$}

\end{document}